\documentclass[regno,12pt]{amsart}
\usepackage{amsmath,amscd,amssymb,amsthm,hyperref,bm}

\newcommand{\mysection}[1]{\section{#1}
\setcounter{equation}{0}}

\newtheorem{theorem}{Theorem}[section]
\newtheorem{lemma}[theorem]{Lemma}
\newtheorem{corollary}[theorem]{Corollary}

\theoremstyle{definition}

\newtheorem{definition}[theorem]{Definition}

\theoremstyle{remark}
\newtheorem{remark}[theorem]{Remark}
\newtheorem{example}[theorem]{Example}

\newcommand{\bb}[1]{\mathbb{#1}}
\newcommand{\etext}[1]{\quad\emph{#1}\quad}
\newcommand{\itext}[1]{\quad\text{#1}\quad}
\newcommand{\ti}[1]{\textit{#1}}
\newcommand{\vect}[1]{\bm{#1}}
\newcommand{\er}[1]{\eqref{#1}}

\newcommand{\con}{{\rm const}}

\newcommand{\diam}{{\rm diam}\,}
\newcommand{\dist}{{\rm dist}}

\newcommand{\eq}{\equiv}

\newcommand{\0}{\emptyset}
\newcommand{\8}{\infty}
\newcommand{\ol}{\overline}
\newcommand{\oo}{\overline{\Omega}}
\newcommand{\pa}{\partial}
\newcommand{\po}{\partial\Omega}
\newcommand{\real}{{\rm Re}\,}
\newcommand{\rn}{{\mathbb R}^n}

\newcommand{\sm}{\setminus}
\newcommand{\su}{\subset}

\newcommand{\tr}{{\rm tr}\,}
\newcommand{\ve}{\varepsilon}
\newcommand{\vp}{\varphi}

\title[Boundary estimates for positive solutions]{Boundary estimates for positive solutions to second order elliptic equations}

\thanks{The work  was partially supported by NSF Grant DMS-9971052}
\author{Mikhail V. Safonov}

\address{School of Mathematics, University of Minnesota}
\email{safonov@math.umn.edu}


\begin{document}

\begin{abstract}
Consider positive solutions to second order elliptic equations with measurable coefficients in a bounded domain, which vanish on a portion of the boundary. We give simple necessary and sufficient geometric conditions on the domain, which guarantee the Hopf-Oleinik type estimates and the boundary Lipschitz estimates for solutions. These conditions are sharp even for harmonic functions.
\end{abstract}

\maketitle

\mysection{Introduction. Formulation of main results}\label{S.1}

Let $\Omega$ be a bounded open set in $\rn$. Consider a second order elliptic operator
    \begin{equation}\label{1.1}
        Lu:= \sum_{i,j} a_{ij}D_{ij}u + \sum_i b_i D_iu
    \end{equation}
in $\Omega$, where $D_iu:=\pa u/\pa x_i,\; D_{ij}u:=D_iD_j u,\; a_{ij}=a_{ji}\in L^{\8}(\rn),\;
b_i\in L^{\8}(\rn)$, and $a_{ij}$ satisfy the \ti{uniform ellipticity condition}
    \begin{equation}\label{1.2}
        \nu |\xi|^2\le \sum_{i,j}a_{ij}\xi_i\xi_j\le \nu^{-1}|\xi|^2
        \itext{for all} \xi=(\xi_1,\ldots,\xi_n)\in\rn,
    \end{equation}
with a constant $\nu\in (0,1]$. In 1952, E. Hopf \cite{H52} and O.A. Oleinik \cite{O52} independently proved the following \ti{boundary point lemma}.

\begin{lemma}\label{L1.1}
    Suppose that $\Omega$ satisfies an \emph{interior sphere condition} at $x_0\in\po$, i.e. there exists a ball
    \[B:=B_{r_0}(y_0):=\{x\in\rn:\;|x-y_0|<r_0\}\su\Omega,\]
     with $x_0\in (\po)\cap (\pa B)$. Then for any function $u\in C^2(\Omega)\cap C(\oo)$ satisfying $\,u>0,\;Lu\le 0\,$ in $\Omega$, and $u(x_0)=0$, we have
    \begin{equation}\label{1.3}
        \liminf_{t\to 0^+} \frac{u(x_0+t\vect{l})}{t}>0.
    \end{equation}
     where $\vect{l}$ is an arbitrary interior vector to $B$ at  the point $x_0$, which means $x_0+t\vect{l}\in B$ for all $t$ in an interval $(0,t_0)$.
\end{lemma}

In a particular case when $L=\Delta$ - the Laplacian, this result was established in 1910 by M.S. Zaremba \cite{Z10}. In the beginning of 1930s, G. Giraud \cite{G33} has got a similar result for domains $\Omega$ with the boundary $\po\in C^{1,\alpha},\,0<\alpha<1$, and operators $L$ with coefficients satisfying some continuity assumptions. See bibliographical notes in \cite{PW}, Ch. 2, and \cite{GT}, Ch. 3, for early references on this subject.
\smallskip

On the other hand, it is well known (see, e.g. \cite{CH}, IV.7.3) that an \emph{exterior sphere condition} at $x_0\in\po$, together with the boundary condition $u=0$ near $x_0$, guarantees the boundedness of the ratio $u(x)/|x-x_0|$ in $\Omega$. In a ``model'' case, this property can be formulates as follows.

\begin{lemma}\label{L1.2}
     Suppose that $\Omega$ satisfies an \emph{exterior sphere condition} at a point $x_0\in\po$, i.e. there exists a ball
     $B:=B_{r_0}(y_0)$, such that $\Omega\cap B=\0$, and $x_0\in (\po)\cap (\pa B)$.
     Let $u\in C^2(\Omega)\cap C(\oo)$ satisfy $\,u>0,\;Lu\ge 0\,$ in $\Omega$, and
     \[ u=0 \etext{on} (\po)\cap(B_{\ve_0}(x_0),
     \itext{where} \ve_0=\con>0.\]
     Then
     \begin{equation}\label{1.4}
        \sup_{\Omega} \frac{u(x)}{|x-x_0|}<\8.
     \end{equation}
\end{lemma}

The proofs of Lemmas \ref{L1.1} and \ref{L1.2} and their generalizations are usually based on the classical \emph{comparison principle} (\cite{GT}, Theorem 3.3).

\begin{theorem}[Comparison principle]\label{T1.3}
    Let $\Omega$ be a bounded open set in $\rn$, and let $u_1,u_2$ be functions in $C^2(\Omega)\cap C(\oo)$ satisfying $Lu_1\ge Lu_2$ in $\Omega$, and $u_1\le u_2$ on $\po$. Then $u_1\le u_2$ in $\Omega$.
\end{theorem}

We give short proofs of Lemmas \ref{L1.1} and \ref{L1.2}, which contain some elements of the proofs of our main results, Theorems \ref{T1.8}  and \ref{T1.9}. For this purpose, we need the following elementary lemma, which will also be useful later, in the proof of Lemma \ref{L2.3}.

\begin{lemma}\label{L1.4}
    The functions $v(x):=|x|^{-\lambda}$ satisfies the inequality\\
     $\sum a_{ij}D_{ij}v\ge 0$ in $\rn\sm\{0\}$, provided the constant $\lambda=\lambda(n,\nu)>0$ is large enough.
\end{lemma}

\begin{proof}
    We have
    \[ \begin{split}
      \sum_{i,j}a_{ij}D_{ij}\big(|x|^{-\lambda}\big) &
      =\lambda |x|^{-\lambda-2}\cdot\bigg[(\lambda+2)\sum_{i,j}\frac{a_{ij}x_ix_j}{|x|^2}
      -\tr a \bigg]\\
        & \ge \lambda |x|^{-\lambda-2}\cdot \big[(\lambda+2)\nu- n\nu^{-1}\big]
        \ge 0\itext{for} x\ne 0,
    \end{split} \]
provided $\lambda>0$ and $\lambda+2\ge n\nu^{-2}$.
\end{proof}

\begin{remark}\label{R1.5}
    The previous lemma says that $L\big(|x|^{-\lambda}\big)\ge 0$ for $x\ne 0$, where $L$ is an operator in \er{1.1} with $b_i\eq 0$. One can easily adjust the proof of this lemma to the case $|b_i|\le K=\con$, with $\lambda=\lambda(n,\nu,K,\diam\Omega)>0$.
\end{remark}

\noindent
\emph{Proof of Lemma \ref{L1.1}.}
    We have $u\ge c=\con>0$ on the set $\pa B_{r_0/2}(y_0)$, which is a compact subset of $\Omega$.
    Following the argument in \S 1.3 of the book by E.M. Landis \cite{L71}, consider the function
    \[ u_1(x):=c_1\,\big(|x-y_0|^{-\lambda}-r_0^{-\lambda}\big)
    \itext{in} \Omega_1:=B_{r_0}(y_0)\sm B_{r_0/2}(y_0)
    \su\Omega, \]
    where $c_1:=(2^{\gamma}-1)^{-1}r_0^{\gamma}\,c>0$.
    Then $u_1=c\le u$ on $\pa B_{r_0/2}(y_0)$, and $u_1=0\le u$ on $\pa B_{r_0}(y_0)$, i.e. $u_1\le u$ on $\pa\Omega_1$. Moreover, $Lu_1\ge 0\ge Lu$ in $\Omega_1$. By the comparison principle, we have $u_1\le u$ in $\Omega_1$. It is easy to see that \er{1.3} holds true for the function $u_1$, hence it is also true for the given function $u$.
\hfill $\Box$
\medskip

\noindent
\emph{Proof of Lemma \ref{L1.2}.}
    We adjust the argument in \S IV.7.3 of the book by R. Courant and D. Hilbert \cite{CH}. Replacing the ball $B$ by a smaller ball if necessary, one can assume that it lies at a positive distance from $(\po)\sm B_{\ve_0}(x_0)$. Then it is possible to choose a constant $R_0>r_0$ close to $r_0$, such that the set $(\po)\cap \big(B_{R_0}(y_0)\sm B_{r_0}(y_0)\big)$ is a subset of $(\po)\cap B_{\ve_0}(x_0)$.
    Consider the function
    \[ u_2(x):=c_2\,\big(r_0^{-\lambda}-|x-y_0|^{-\lambda}\big)
    \itext{in} \Omega_2:=\Omega\cap\big(B_{R_0}(y_0)\sm B_{r_0}(y_0)\big).\]
    Here $c_2>0$ is a large enough constant, such that
    \[ u\le c_2\,\big(r_0^{-\lambda}-R_0^{-\lambda}\big)=u_2
    \itext{on} \Omega\cap \pa B_{R_0}(y_0).\]
    On the remaining part of $\po_1$, which is a subset of $(\po)\cap B_{\ve_0}(x_0)$, we have $u=0\le u_2$. This means $u\le u_2$ on $\po_2$. Moreover, $Lu\ge 0\ge Lu_2$ in $\Omega_2$.  By the comparison principle, we have $u\le u_2$ in $\Omega_2$. Since $u_2$ is a Lipschitz function on $\Omega_2$, and $u_2(x_0)=0$, the ratio
    \[ \frac{u(x)}{|x-x_0|}\le \frac{u_2(x)}{|x-x_0|}\le N=\con
    \itext{in} \Omega_2.\]
    On the complementary set $\Omega\sm \Omega_2$, the function $u\in C(\oo)$ is bounded, and $|x-x_0|\ge R_0-r_0>0$. This implies the desired estimate \er{1.4}.
\hfill $\Box$
\medskip

In the formulations of Lemmas \ref{L1.1} and \ref{L1.2}, one cannot replace an exterior or interior sphere condition by a corresponding cone condition, as the following simple example shows.

\begin{example}\label{E1.6}
    (i) Fix a constant $\theta_1\in (0,\pi/2)$ and denote
    \[ \Omega_1:=\{ x=(x_1,x_2)\in {\bb R}^2:
    \;|x|<1,\; x_2>K\cdot |x_1|\},\]
    where $K:=\cot \theta_1>0$. In the polar coordinates $x_1=\rho\sin\theta,\; x_2=\rho\cos\theta$, we have
    \[ \Omega_1:=\{0<\rho<1,\;|\theta|<\theta_1\},
    \itext{and} z:=ix_1+x_2=\rho e^{i\theta}.\]
    The function
    \[ u_1(x_1,x_2):=\real\big(z^{\gamma_1}\big)
    =\rho^{\gamma_1}\cos(\gamma_1\theta),
    \itext{where} \gamma_1:=\frac{\pi}{2\theta_1}>1,\]
    belongs to $C^{\8}(\Omega_1)\cap C(\ol{\Omega_1})$ and satisfies $\,u_1>0,\,\Delta u_1=0\,$ in $\Omega_1$, and $u_1(0)=0$.
    It is easy to see that $u_1$ does not satisfy the strict inequality \er{1.3} (we have an equality) at the point $x_0=0\in\po_1$, where $\vect{l}$ is an arbitrary interior vector to $\Omega_1$.
    \smallskip

    (ii) The set
    \[ \Omega_2:=\{ x=(x_1,x_2)\in {\bb R}^2:
    \;|x|<1,\; x_2>-K\cdot |x_1|\},\]
    can be described in a similar way with $\theta_2:=\pi-\theta_1\in (\pi/2,\pi)$ in place of $\theta_1$. The function
    \[ u_2(x_1,x_2):=\real\big(z^{\gamma_2}\big)
    =\rho^{\gamma_2}\cos(\gamma_2\theta),
    \itext{where} \gamma_2:=\frac{\pi}{2\theta_2}\in (0,1),\]
    belongs to $C^{\8}(\Omega_2)\cap C(\ol{\Omega_2})$ and satisfies $\,u_2>0,\,\Delta u_2=0\,$ in $\Omega_2$, and $u_2\eq 0$ on $(\po_2)\cap B_1(0)$. Obviously, the ratio $u_2(x)/|x|$ is unbounded on $\Omega_2$, i.e. \er{1.4} fails at the point $x_0=0\in\po_2$.
\end{example}

Now consider a more general situation, when a ball $B$ in Lemmas \ref{L1.1} and \ref{L1.2} is replaced by a body of rotation  $Q$.

\begin{definition}\label{D1.7}
    Let a constant $r_0>0$ be given, and let $\psi(r)$ be a non-negative, non-decreasing function on $[0,r_0]$, with $\psi(r_0)<r_0$.  Define
    \begin{equation}\label{1.5}
        Q:=\{x=(x',x_n)\in\rn:\;|x'|<r_0,\;0<x_n-\psi(|x'|)<r_0 \}.
    \end{equation}

    (i) We say that an open set $\Omega\su\rn$ satisfies an \emph{interior $Q$-condition} at a point $x_0\in\po$ if $\Omega$ contains a body which is congruent with $Q$ with vertex at $x_0$. This means that in an appropriate coordinate system, we have $Q\su\Omega$, and $x_0=0\in(\po)\cap(\pa Q)$.

    (ii) We say that an open set $\Omega\su\rn$ satisfies an \emph{exterior $Q$-condition} at a point $x_0\in\po$ if its complement $\,\Omega^c:=\rn\sm\ol{\Omega}\,$ satisfies an interior $Q$-condition at $x_0$.
\end{definition}

Our main results are contained in Theorems \ref{T1.8}--\ref{T1.11} below. Theorems \ref{T1.8} and \ref{T1.9} can be considered as generalizations of Lemmas \ref{L1.1} and \ref{L1.2} correspondingly, when instead of (exterior or interior) sphere conditions we impose $Q$-conditions with
    \begin{equation}\label{1.6}
        I(\psi):=\int_0^{r_0}\frac{\psi(r)\,dr}{r^2}<\8.
    \end{equation}
Without loss of generality, we assume that the coordinate system is chosen in such a way that $x_0=0\in\po$, $Q\su\Omega$ if $\Omega$ satisfies an interior $Q$-condition, and $-Q:=\{x\in\rn:\; -x\in Q\}\su \Omega^c:=\rn\sm \Omega$ if $\Omega$ satisfies an exterior $Q$-condition. Note that sphere conditions are equivalent to $Q$-condition with $\psi(r)=cr^2,\,c=\con>0$. In this case $I(\psi)<\8$ automatically. We prove Theorems \ref{T1.8} and \ref{T1.9} in Section \ref{S.3}. Another two theorems, Theorems \ref{T1.10} and \ref{T1.11},  are given here just for completeness, without proofs. They claim that the assumption $I(\psi)<\8$ is sharp: if $I(\psi)=\8$, then the estimates in Lemmas \ref{L1.1} and \ref{L1.2} fail. Example \ref{E1.6} can serve as a clear demonstration of this fact for $\psi(r)=Kr$.

In Theorems \ref{T1.8}--\ref{T1.11}, we assume that $u\in C^2(\Omega)\cap C(\oo)$ is a positive solution of the inequality $Lu\le 0$ or $Lu\ge 0$ in $\Omega$, where $Lu:=\sum a_{ij}D_{ij}u$ has the form \er{1.1}, \er{1.2}, with $b_i\eq 0$. Combining our techniques with others, especially those in the paper by O. A. Ladyzhenskaya and N. N. Ural'tseva \cite{LU88}, one can extend the results in Theorems \ref{T1.8}--\ref{T1.11} to more general operators $L$ in \er{1.1} with $b_i\in L^q,\;q>n$. We plan to do it in our subsequent work. In particular, proofs of Theorems \ref{T1.10} and \ref{T1.11} will be presented in a more general setting. On the other hand, Example \ref{E1.12} below shows that in the case $b_i\in L^n$ all the estimates under considerations fail even for flat boundary, when $\psi\eq 0$.
Here we restrict ourselves to the case $b_j\eq 0$ in order to expose our method in its ``pure'' form.

\begin{theorem}\label{T1.8}
    Suppose that $\Omega$ satisfies an interior $Q$-condition: $Q\su\Omega$, with $I(\psi)<\8$, and $0\in\po$.  Then for any function $u\in C^2(\Omega)\cap C(\oo)$ satisfying $\,u>0,\;Lu\le 0\,$ in $\Omega$, and $u(0)=0$, we have
     \begin{equation}\label{1.7}
        \liminf_{t\to 0^+}\, t^{-1} u(t\,\vect{l})>0
        \itext{for each} \vect{l}\in
        \rn_+:=\{x\in\rn:\;x_n>0\}.
    \end{equation}

     Note that from $I(\psi)<\8$ it follows that $t\vect{l}\in Q\su\Omega$ for small $t>0$ (Corollary \ref{C3.2} below), so that $u(t\vect{l})$ in \er{1.7} is well defined.
\end{theorem}

\begin{theorem}\label{T1.9}
    Suppose that $\Omega$ satisfies an exterior $Q$-condition: $-Q\su\Omega^c$, with $I(\psi)<\8$, and $0\in\po$.  Then for any function $u\in C^2(\Omega)\cap C(\oo)$ satisfying $\,u>0,\;Lu\ge 0\,$ in $\Omega$, and $u=0$ on $(\po)\cap B_{r_0}(0)$, we have
    \begin{equation}\label{1.8}
        M(r_0):=\sup_{\Omega\cap B_{r_0}(0)}
        \frac{u(x)}{|x|}<\8.
    \end{equation}
\end{theorem}

The notation $M(r)$ is also used in the following
\begin{theorem}\label{T1.10}
    Suppose that $\Omega\cap B_{r_0}(0)\su Q$, with $I(\psi)=\8$, and $0\in\po$. Then for any function $u\in C^2(\Omega)\cap C(\oo)$ satisfying $\,u>0,\;Lu\ge 0\,$ in $\Omega$, and $u=0$ on $(\po)\cap B_{r_0}(0)$, we have $M(r)\to 0$ as $r\to 0^+$.
    Obviously, in this case the estimate \er{1.7} fails.
\end{theorem}

\begin{theorem}\label{T1.11}
    Suppose that $\Omega^c\cap B_{r_0}(0)\su -Q$, with $I(\psi)=\8$, and $0\in\po$. Then for any function $u\in C^2(\Omega)\cap C(\oo)$ satisfying $\,u>0,\;Lu\le 0\,$ in $\Omega$, and $u=0$ on $(\po)\cap B_{r_0}(0)$, we have
         \begin{equation}\label{1.9}
        \liminf_{t\to 0^+}\, t^{-1} u(t\,\vect{l})=\8
        \itext{for all} \vect{l}\in \rn_+.
    \end{equation}
\end{theorem}

In 1969--1970, similar facts were  established by B.N. Khimchenko, first in the case $L=\Delta$ \cite{Kh69}, and then for general elliptic operators $L\,$ \cite{Kh70}, under the additional assumption $\psi''\ge 0$ (in these two papers, the same author's name is spelled slightly differently). Further, is a series of joint papers by L.I. Kamynin and B.N. Khimchenko (see \cite{KH80} and references therein), these results were extended to the parabolic and degenerate elliptic equations, under a different assumption $\psi(r)=r\psi_1(r)$ with $\psi'_1\ge 0,\;\psi''_1\le 0$. Each of these assumptions, as well as our assumption \er{1.6}, holds true for $\psi(r):=r^{1+\alpha},\,0<\alpha<1$, so that the above mentioned result by G. Giraud \cite{G33} for $\po\in C^{1,\alpha}$ is extended to general operators $L$ with bounded measurable coefficients. This case is also covered in the paper \cite{L85} by Gary M. Lieberman, in which $\po$ has a Dini continuous normal.

In the papers \cite{H52}, \cite{O52}, \cite{Kh69}, \cite{Kh70}, \cite{KH80}, \cite{L85}, and many others, the estimates of such kind are proved by means of special comparison functions, which are constructed in a more or less explicit form. Our method is quite different: it does not use any explicit expressions for comparison functions, and it  does not require additional assumptions on the functions $\psi(r)$ in Definition \ref{D1.7}. Instead, we use the estimates for quotients $u_2/u_1$ of positive solutions of $Lu=0$ in a Lipschitz domain $\Omega$, which vanish on a portion of $\po$. These estimates were proved by Patricia  Bauman in 1982 in her PhD thesis \cite{B82}, and published a bit later in \cite{B84}. Note that some estimates in her paper depend on the modulus of continuity of coefficients $a_{ij}$. However, it is easy to get rid of this additional dependence. In a more general parabolic case, this was done in \cite{FSY}, Theorem 4.3.

We essentially use the fact that $u(x)\eq x_n$ is a solution to the elliptic equation $Lu:=\sum a_{ij}D_{ij}u=0$; this is why we assume $b_i\eq 0$ in \er{1.1}. Note that the estimates for the quotients $u_2/u_1$ are also true for solutions to the equations in the \emph{divergence} form $Lu:=\sum D_i(a_{ij}D_ju)=0$ (see \cite{CFMS}), but they are not helpful here, because linear functions do not satisfy such equations in general, and in fact, the Hopf-Oleinik estimate \er{1.3} fails even when the boundary is flat (see \cite{GT}, Problem 3.9).

\begin{example}\label{E1.12}
    Consider the functions
    \[ u_1(x):=\frac{x_n}{|\ln|x||}
    \etext{and} u_2(x):=x_n\cdot |\ln|x||\]
    in the cylinder $Q:=\{x=(x',x_n)\in\rn:\;|x'|<1/2,\;0<x_n<1/2\}$,
    extended as $u_1=u_2=0$ on $(\pa Q)\cap \{x_n=0\}$. Then each of these two functions can be considered as a solution to the equation
    \[ \Delta u+\vect{b}\cdot Du:= \Delta u+\sum_i b_iD_iu=0
    \etext{in} Q,\]
    where the vector function $\vect{b}:=-\Delta u\cdot |Du|^{-2} Du$ satisfies
    \[ |\vect{b}| = \frac{|\Delta u|}{|Du|}
    \le \frac{\con}{|x|\cdot |\ln|x||}
    \in L^n(Q)
    \etext{for}n\ge 2.\]
    However, the left side of \er{1.7} is $0$ for $u=u_1$, and the left side of \er{1.8} is $\8$ for $u=u_2$.
\end{example}

In Section \ref{S.2}, we bring together, in a convenient form, some basic facts, including the estimated for the quotients $u_2/u_1$ of positive solutions, which are essential for our approach. Finally, in Section \ref{S.3}, we prove Theorems \ref{T1.8} and \ref{T1.9}.
\medskip

\emph{Notations}.  We use notations $N$ and $c$ for various positive constants depending only on the prescribed constants, such as $n,\,\nu$, etc., which do not depend on smoothness of coefficients $a_{ij}$. These constants may be different in different expression.
The expression $A:=B$ or $B=:A$ means ``$A=B$ by definition''.

$B_r(x_0):=\{x\in\rn:\,|x-x_0|<r\}$ is a ball of radius $r>0$ centered at $x_0\in\rn$. $\rn_+:=\{x=(x_1,\ldots,x_n)\in\rn:\,x_n>0\}$.
\medskip

\emph{Acknowledgements}. The author thanks N. V. Krylov, N. N. Ural'tseva, and H. F. Weinberger for very useful discussion of results in this paper.

\mysection{Auxiliary statements}\label{S.2}

In the rest of this paper, $Lu:=\sum a_{ij}D_{ij}u$ with $a_{ij}=a_{ji}\in L^{\8}$ satisfying the ellipticity condition \er{1.2} with a constant $\nu\in (0,1]$. Note that the results in this section are valid for more general operators $L$ in \er{1.1}, which include the lower order terms $\sum b_iD_iu$ with $b_i\in L^{\8}$. In this case, the constants $N$ and $c$ depend also on the upper bounds for $|b_i|$.

The following theorem was proved by N. V. Krylov and the author \cite{KS80}, \cite{S80} (see also \cite{K}, Theorem IV.2.8, and \cite{GT}, Corollary 9.25).

\begin{theorem}[Interior Harnack inequality]\label{T2.1}
    Let $\Omega$ be a bounded domain in $\rn$, such that the set
    \begin{equation}\label{2.1}
        \Omega^{\delta}:=\{x\in\Omega:\;\dist(x,\po)>\delta\}
    \end{equation}
is connected, where $\delta=\con>0$. Then
    \begin{equation}\label{2.2}
        \sup_{\Omega^{\delta}}u\le N\cdot \inf_{\Omega^{\delta}}u,
    \end{equation}
    with a constant $N$ depending only on $n,\nu$, and $\delta/\diam \Omega$.
\end{theorem}

\begin{proof}
    In its standard form, the Harnack inequality is formulated for two concentric balls in place of $\Omega^{\delta}$ and $\Omega$, e.g. for $B_{R/8}$ and $B_R$ in \cite{S80}, Theorem 3.1. In general case, fix $x,y\in\Omega^{\delta}$, and choose a sequence $x^{(0)}=x,x^{(1)},\ldots,x^{(m)}=y$ in $\Omega^{\delta}$ such that $|x^{(k-1)}-x^{(k)}|<\delta/8$ for $k=1,2,\ldots,m$. One can do it in such a way that $m$ does  not exceed a constant $m_0$ depending only on $n$ and $\delta/\diam \Omega$. Then applying the ``standard'' Harnack inequality with $R:=\delta$, we get
    \[ u(x^{(k-1)})\le N_1 u(x^{(k)})
    \itext{for} k=1,2,\ldots,m,\]
    where $N_1=N_1(n,\nu)\ge 1$. Therefore,
    \[ u(x)=u(x^{(0)})\le N_1u(x^{(1)})\le \ldots \le N_1^m u(x^{(m)})=N_1^m u(y),\]
    and the desired estimate \er{2.2} follows with $N:=N_1^{m_0}$.
\end{proof}

The following lemma will help us to reduce the proofs of our main results for operators $Lu:=\sum a_{ij}D_{ij}u$ to the case $a_{ij}\in C^{\8}$. We can assume that $a_{ij}$ are defined on the whole space $\rn$. Consider the convolutions $a_{ij}^{\ve}:=a_{ij}*\eta^{\ve}$ with kernels $\eta^{\ve}$ such that
    \[ 0\le \eta^{\ve}\in C^{\8}(\rn),\quad
    \eta^{\ve}(x)\eq 0\itext{for} |x|\ge \ve,
    \itext{and} \int\eta^{\ve}(x)\,dx=1.\]
Then $a_{ij}^{\ve}\in C^{\8}(\rn),\; a_{ij}^{\ve}=a_{ji}^{\ve}$ satisfy \er{1.2}, and moreover,
\begin{equation}\label{2.3}
    a_{ij}^{\ve}\to a_{ij}\itext{as}\ve\to 0^+
    \itext{a.e. in}\Omega.
\end{equation}
This convergence follows from the properties of the Lebesgue sets (see \cite{St70}, Sec. I.1.8).

\begin{lemma}[Approximation lemma]\label{L2.2}
    Let $\Omega$ be a bounded open set in $\rn$ satisfying an exterior cone condition at each point $x_0\in\po$, i.e. an exterior $Q$-condition in Definition \ref{D1.7} with \begin{equation}\label{2.4}
    Q:=\{x=(x',x_n):\;|x|<r_0,\;x_n>K\cdot |x'|\}\end{equation}
    with constants $K>0$ and $r_0>0$. Let $u$ be a function in $C^2(\Omega)\cap C(\oo)$ satisfying $Lu:=\sum a_{ij}D_{ij}u\le 0$ in $\Omega$. For $\ve>0$, consider the above approximations of $a_{ij}$ by functions $a_{ij}^{\ve}\in C^{\8}$, which satisfy \er{1.2} and \er{2.2}, and let $u^{\ve}$ be a unique solution to the problem
    \begin{equation}\label{2.5}
        L^{\ve}u^{\ve}:=\sum_{i,j} a_{ij}^{\ve}D_{ij}u^{\ve}=0
        \itext{in}\Omega,\qquad
        u^{\ve}=u\etext{on}\po,
    \end{equation}
    in the class $C^{\8}(\Omega)\cap C(\oo)$. Then
    \begin{equation}\label{2.6}
        \sup_{\Omega} (u^{\ve}-u)\to 0
        \etext{as} \ve\to 0^+.
    \end{equation}

    If $Lu=0$ in $\Omega$, then $u^{\ve}\to u$ as $\,\ve\to 0^+$ uniformly on $\Omega$.
\end{lemma}

Note that the existence of a solution $u^{\ve}\in C^{\8}(\Omega)\cap C(\oo)$ to the problem \er{2.5} (under an exterior cone condition) follows from the results by K. Miller \cite{M67}.

\begin{proof}
    From the arguments in the proof of Theorem 3 in \cite{M67} it follows that
        \[\sup_{\Omega\cap B_{\delta}(x_0)}
        |u^{\ve}(x)-u(x_0)|
        \le \omega(\delta)\to 0
        \itext{as} \delta\to 0^+,\]
    uniformly with respect to $x_0\in\po$ and $\ve>0$. Since $u\in C(\oo)$, this property also holds true for $u(x)$ in place of $u^{\ve}(x)$. By the triangle inequality, we get
    \begin{equation}\label{2.7}
        \sup_{\Omega\sm\Omega^{\delta}} |u^{\ve}-u|\le 2\omega(\delta),
    \end{equation}
    where $\Omega^{\delta}$ is defined in \er{2.1}. Moreover, since $L^{\ve}u^{\ve}=0\ge Lu$, we also have
    \[ L^{\ve}(u^{\ve}-u)\ge f^{\ve}:=(L-L^{\ve})u
    :=\sum_{i,j}(a_{ij}-a_{ij}^{\ve})D_{ij}u.\]
    Now we can use the A.D. Aleksandrov type estimate (see \cite{A67} or \cite{GT}, Theorem 9.1):
    \[\sup_{\Omega^{\delta}} (u^{\ve}-u)
    \le \sup_{\po^{\delta}} (u^{\ve}-u)
    +N\cdot ||f^{\ve}||_{L^n(\Omega^{\delta})}.\]
    By virtue of \er{2.7}, this yields
    \[ \sup_{\Omega} (u^{\ve}-u)
    \le 2\omega(\delta)+N\cdot ||f^{\ve}||_{L^n(\Omega^{\delta})}.\]
    Since $D_{ij}u$ are bounded on $\Omega^{\delta}$, and $a_{ij}^{\ve}\to a_{ij}$ a.e., the last term converges to $0$ as $\ve\to 0^+$. Hence
    \[ 0\le \limsup_{\ve\to 0^+}\,\sup_{\Omega} (u^{\ve}-u)
    \le 2\omega(\delta).\]
    The desired property \er{2.6} follows by sending $\delta$ to $0$.

    In the case $Lu=0$, we can apply \er{2.6} to both functions $u$ and $-u$, which gives the uniform convergence of $u^{\ve}$ to $u$ on $\Omega$.
\end{proof}

We also need a lower estimate for positive supersolutions in $\Omega$, which are strictly positive on a Lipschitz  portion of the boundary $\po$. For the proof of this estimate, it is convenient to replace the Lipschitz property of $\po$ by a weaker assumption \er{2.9} below.

\begin{lemma}\label{L2.3}
    Let $\Omega$ be a bounded domain in $\rn$, and let $u\in C^2(\Omega)\cap C(\oo)$ satisfy $u>0,\,Lu\le 0$ in $\Omega$. Suppose that
    \begin{equation}\label{2.8}
        u\ge\mu=\con\etext{on} (\po)\cap B_{r_0}(x_0),
    \end{equation}
     where $x_0\in\po$ and $r_0>0$ is a given constant. Moreover, let $\,\delta>0$ be a constant such that the set $\,\Omega^{\delta}$ in \er{2.1} is connected, and there are balls
    \begin{equation}\label{2.9}
        B_{\delta}(y_0)\su \Omega^c\cap B_{r_0/2}(x_0)
        \itext{and}
        B_{\delta}(z_0)\su \Omega\cap B_{r_0/2}(x_0).
    \end{equation}
    Then
    \begin{equation}\label{2.10}
        u\ge c\mu\etext{in} \Omega^{\delta},
        \etext{where} c=c(n,\nu,\delta/\diam\Omega)>0.
    \end{equation}
\end{lemma}

\begin{proof} \emph{Step 1.} From \er{2.9} it follows that $\delta\le r_0/4$, and the balls $B_{3\delta}(y_0)$ and $B_{3\delta}(z_0)$ are contained in $B_{r_0}(x_0)$. Therefore, same is true for $B_{3\delta}(y)$, and by \er{2.8}, $u\ge\mu$ on $(\po)\cap B_{3\delta}(y)$ for each $y$ in the segment $[y_0,z_0]$.

Next, choose a sequence of points $y_0,y_1,\ldots,y_m=z_0$ in $[y_0,z_0]$, such that $|y_{k+1}-y_k|\le \delta$ for all $k=0,1,\ldots,m-1$. Obviously, we can assume that $m$ does not exceed a constant $m_1$ depending only on $\delta/\diam\Omega$. We claim that
\begin{equation}\label{2.11}
    u\ge\theta^k\mu\itext{in}\Omega\cap B_{\delta}(y_k)
    \itext{for}k=0,1,\ldots,m,
\end{equation}
with a constant $\theta=\theta(n,\nu)\in (0,1)$, to be specified later. Here we impose a natural agreement that \er{2.11} is true automatically if $\Omega\cap B_{\delta}(y_k)$ is empty, which is the case if $k=0$. In order to use  induction, we only need to prove \er{2.11} with $k+1$ in place of $k$, based on the assumption that it is true for some $k<m$. For this purpose, we compare the function $u(x)$ with
    \[ v_k(x):=\theta^k\mu\cdot
    \frac{|x-y_k|^{-\gamma}-(3\delta)^{-\gamma}}
    {\delta^{-\gamma}-(3\delta)^{-\gamma}}
    \itext{in}
    \Omega_k:=\Omega\cap
    \big(B_{3\delta}(y_k)\sm B_{\delta}(y_k)\big),\]
where $\gamma=\gamma(n,\nu)>0$ is a constant in Lemma \ref{L1.4}. Of course, we can skip this part if $\Omega_k$ is empty. By this lemma, $Lv_k\ge 0\ge Lu$ in $\Omega_k$. Moreover, \er{2.11} implies $u\ge\theta^k\mu=v_k$ on $(\po_k)\cap \pa B_{\delta}(y_k)$. We also have $u\ge 0=v_k$ on $(\po_k)\cap \pa B_{3\delta}(y_k)$, and by \er{2.8}, $u\ge\mu\ge v_k$ on the remaining part of $\po_k$. By the comparison principle, $u\ge v_k$ in $\Omega_k$. Together with \er{2.11}, this gives us
    \[u\ge \theta^{k+1}\mu\itext{in}
    \Omega\cap B_{2\delta}(y_k),\itext{if}
    \theta:=\frac{(3/2)^{\gamma}-1}{3^{\gamma}-1}\in (0,1).\]
Finally, $|y_{k+1}-y_k|\le\delta$ implies that the set $\Omega\cap B_{\delta}(y_{k+1})$ is contained in $\Omega\cap B_{2\delta}(y_k)$, so that the inequality in \er{2.11} holds true for $k+1$. By induction, the proof of \er{2.11} is complete. In particular, taking $k=m\le m_1$, we get
    \begin{equation}\label{2.12}
    u\ge  c_1\mu\itext{on} B_{\delta}(z_0),
    \itext{where} c_1:=\theta^{m_1}>0.
    \end{equation}

\emph{Step 2.} For an arbitrary point $z\in \Omega^{\delta}$, and choose a sequence of points $z_0,z_1,\ldots,z_m=z$ in $\Omega^{\delta}$, such that
$|z_{k+1}-z_k|\le \delta_1:=\delta /3$ for all $k=0,1,\ldots,m-1$. Here we can assume that
$m\le m_2=m_2(n,\delta/\diam\Omega)$. Similarly to \er{2.11}, with $z_k$ in place of $y_k$ and $\delta_1$ in place of $\delta$, and some simplifications because of the property $B_{3\delta_1}(z_k)=B_{\delta}(z_k)\su\Omega$, we obtain
    \[ u\ge\theta^k c_1\mu\itext{in}B_{\delta_1}(z_k)
    \itext{for}k=0,1,\ldots,m.\]
In particular,
$u(z)=u(z_m)\ge \theta^m c_1\mu \ge \theta^{m_2} c_1\mu$. Since $z$ is an arbitrary point in $\Omega^{\delta}$, the desired estimate \er{2.10} is proved with $c:=\theta^{m_2} c_1=\theta^{m_1+m_2}$.
\end{proof}

The following theorem, which is due P. Bauman (see \cite{B84}, Theorem 2.1), is the main tool in our approach.

\begin{theorem}[Comparison theorem]\label{T2.4}
    Let $\vp$ be a Lipschitz continuous function on ${\bb R}^{n-1}$:
    \[ |\vp(x')-\vp(y')|\le K\cdot |x'-y'|\etext{for all} x',y'\in {\bb R}^{n-1},\]
    with $K=\con\ge 0$, and $\vp(0)=0$. For $\,r>0$, define
    \[\Omega_r:=\{x=(x',x_n)\in\rn:\;|x'|<r,\;0<x_n-\vp(x')<r\},\]
    and $\Gamma_r:=(\pa \Omega_r)\cap\{x_n=\vp(x')\}$. Let $u,v$ be functions in $C^2(\Omega_{2r})\cap C(\ol{\Omega_{2r}})$ satisfying
    \[ u>0,\;v>0,\quad Lu=Lv=0\etext{in} \Omega_{2r},\]
    and $u=v=0$ on $\Gamma_{2r}$. Then
    \begin{equation}\label{2.13}
        \sup_{\Omega_r}\frac{u}{v}\le N\cdot \frac{u(0,r)}{v(0,r)},
        \etext{where} N=N(n,\nu,K)>0.
    \end{equation}
\end{theorem}

\begin{corollary}\label{C2.5}
    Under the assumptions of the previous theorem, we also have
    \begin{equation}\label{2.14}
    \frac{u(0,r)}{v(0,r)}\le N\cdot
    \inf_{\Omega_r}\frac{u}{v},
    \etext{where} N=N(n,\nu,K)>0.
    \end{equation}
\end{corollary}
\begin{proof}
    Obviously, we can interchange $u$ and $v$ in \er{2.13}, and then \er{2.14} follows from an elementary relation  $\inf(u/v)=\big(\sup(v/u)\big)^{-1}$.
\end{proof}

\begin{remark}\label{R2.6}
    In \cite{B84}, this theorem was proved with $\Omega_{8r},\,\Gamma_{8r}$ in place of $\Omega_r,\,\Gamma_r$ correspondingly. In order to apply this fact to the proof of \er{2.13}, consider separately each of two possible cases for $x=(x',x_n)\in\Omega_r$: (i) $x_n-\vp(x')<r/8$ and (ii) $x_n-\vp(x')\ge r/8$. In the case (i), from \cite{B84}, after obvious change of notations, it follows
    \[ \frac{u(x)}{v(x)}\le N_1(n,\nu,K)\cdot \frac{u(x',\vp(x')+r/8)}{v(x',\vp(x')+r/8)},\]
    and then by the Harnack inequality, Theorem \ref{2.1},
    \[ \frac{u(x',\vp(x')+r/8)}{v(x',\vp(x')+r/8)}\le  N_2(n,\nu,K)\cdot \frac{u(0,r)}{v(0,r)},\]
    so that $u/v\,(x)\le N\cdot u/v(0,r)$ with $N:=N_1N_2$. In the case (ii), we get this estimate with $N:=N_2$ by the Harnack inequality directly. Therefore, \er{2.13} holds true.

    The above argument also shows that in the formulation of Theorem \ref{2.4}, one can replace $2r$ by $cr$ with any absolute constant $c>1$. We will use this observation with $c=3/2$ in order to get the estimate \er{2.15} below.
\end{remark}

\begin{corollary}\label{C2.7}
    The estimate \er{2.13} in Theorem \ref{2.4} remains valid if the condition $v=0$ on $\Gamma_{2r}$ is omitted.
\end{corollary}

\begin{proof}
    Having in mind the approximation lemma (Lemma \ref{L2.3}), we can assume that $a_{ij}$ are smooth. Take a continuous function $g$ on $\pa\Omega_{3r/2}$ such that $0\le g\le v$ on $\pa\Omega_{3r/2}$, $g\eq 0$ on $\Gamma_{3r/2}$ and $g\eq v$ on
    \[\Gamma^*_{3r/2}:=
    \big\{x=(x,,x_n)\in\rn:\; |x'|\le 3r/2,\;x_n-\vp(x')=3r/2 \big\}.\]
    Since $a_{ij}$ are smooth, there exists a solution $v_0\in C^2(\Omega_{3r/2})\cap(\ol{\Omega_{3r/2}})$ of the problem
    \[ Lv_0=0\itext{in}\Omega_{3r/2},\qquad
    v_0=g\itext{on}\pa\Omega_{3r/2}.\]
    By Theorem \ref{2.4}, applied to the functions $u$ and $v_0$ in $\Omega_{3r/2}$,
    \begin{equation}\label{2.15}
        \sup_{\Omega_r}\frac{u}{v_0}\le N\cdot \frac{u(0,r)}{v_0(0,r)}.
    \end{equation}
    Moreover, by the comparison principle, $0\le v_0\le v$ in $\Omega_{3r/2}\supset \Omega_r$, hence we can replace $v_0$ by $v$ in the left side. In the right side, we first apply Lemma \ref{L2.3} to the function $v_0$ in $\Omega_{3r/2}$ with $r_0:=3r/2$ and $x_0:=(0,3r/2)\in\Gamma^*_{3r/2}$, and then the Harnack inequality to the function $v$ in $\Omega_{2r}$. As a result, we get
    \[ v_0(0,r)\ge c_1\mu,\itext{where}
    \mu:=\inf_{\Gamma^*_{3r/2}}v_0=\inf_{\Gamma^*_{3r/2}}v
    \ge c_2\cdot v(0,r),\]
    with positive constants $c_1$ and $c_2$ depending only on $n,\nu$ and $K$. Therefore, from \er{2.15} it follows the desired estimate \er{2.13}.
    \end{proof}

\mysection{Proof of Theorems \ref{T1.8} and \ref{T1.9}} \label{S.3}

First of all, we write the integral condition $I(\psi)<\8$ in \er{1.6} in an equivalent ``discrete'' form.

\begin{lemma}\label{L3.1}
    Let $\psi(r)$ be a non-negative, non-decreasing function on $[0,r_0]$, where $r_0=\con>0$. Then $I(\psi)<\8$ if and only if \begin{equation}\label{3.1}
    \sum_{k=0}^{\8}\frac{h_k}{r_k}<\8\itext{where}
    r_k:=4^{-k}r_0,\quad h_k:=\psi(r_k).
    \end{equation}
\end{lemma}

\begin{proof}
    Since $h_{k+1}\le\psi(r)\le h_k$ on $[r_{k+1},r_k]$, and $r_k-r_{k+1}=3r_{k+1}=3r_k/4$, we obtain
    \[I(\psi)=\sum_{k=0}^{\8}\int_{r_{k+1}}^{r_k}\frac{\psi(r)\,dr}{r^2}
    \ge\sum_{k=0}^{\8}\frac{3r_{k+1}h_{k+1}}{r_k^2}
    =\frac{3}{16} \sum_{k=0}^{\8}\frac{h_{k+1}}{r_{k+1}}.\]
    On the other hand, $I(\psi)\le\sum 3r_{k+1}h_k/r_{k+1}^2
    =12\sum h_k/r_k$. Therefore, $I(\psi)<\8$ if and only if $\sum h_k/r_k<\8$.
\end{proof}

\begin{corollary}\label{C3.2}
    If $I(\psi)<\8$, then for arbitrary constant $K_0>0$, there is a constant $0<R_0\le \min(r_0,h_0)$ such that the set
    \[ V_0:=\{x=(x',x_n):\,|x|<R_0,\;x_n>K_0\,|x'|\}\]
    is contained in $Q$.
\end{corollary}

\begin{proof}
    From $I(\psi)<\8$ it follows $\sum h_k/r_k<\8$, hence $h_k/r_{k+1}=4h_k/r_k\to 0$ as $k\to\8$. Choose an integer $k_0\ge 1$ such that $h_k/r_{k+1}\le K_0$ for all $k\ge k_0$, and set $R_0:=\min(r_{k_0},h_{k_0})$. We claim that that each $x=(x',x_n)\in V_0$ belongs to $Q$. This is obvious if $x'=0$, so we can assume $x'\ne 0$. Then there is an integer $k\ge k_0$ (depending on $x$) such that $r_{k+1}\le |x'|<r_k\le R_0$. This implies
    \[\psi(|x'|)\le \psi(r_k)=:h_k\le K_0\,r_{k+1}\le K_0\,|x'|<x_n.\]
    which means $x\in Q$.
\end{proof}

The next lemma can be considered as a very special case of Theorem \ref{T1.8}. However, this ``model'' case contains the main difficulties, so that Theorem \ref{T1.8} in full generality follows easily by the comparison principle.

\begin{lemma}\label{L3.3}
    Let $Q$ be a set defined in \er{1.5}, where $r_0=\con>0$, and $\psi(r)$ is a non-negative, non-decreasing function on $[0,r_0]$, satisfying the condition $I(\psi)<\8$ in \er{1.6}. Let $v$ be a function in $C^2(Q)\cap C(\ol{Q})$, such that \[v>0,\quad Lv:=\sum_{i,j=1}^n a_{ij}D_{ij}v=0
    \itext{in}Q,\]
    and $v=0$ on $\Gamma:=(\pa Q)\cap \{x_n\in\psi(|x'|)\}$. Then
    \begin{equation}\label{3.2}
    \inf_{0<x_n \le r_0/2} \frac{v(0,x_n)}{x_n}>0.
    \end{equation}

    Note that the non-decreasing function $\psi(r)$ may be discontinuous. In order to guarantee that the set $\Gamma$ is connected, we define $\psi(r):=[\psi(r-),\psi(r+)]$ - the segment whose ends are one-sided limits of $\psi(r')$ as $r'\to r$, subject to restriction $r'<r$ or $r'>r$. Obviously, if $\psi$ is continuous at some point $r$, then this segment is reduced to the corresponding point $\psi(r)$.
\end{lemma}

\begin{proof}
    We assume that the coefficients $a_{ij}$ are smooth functions on $\rn$. The general case follows from the approximation lemma (Lemma \ref{L2.2}), because all the estimates in the proof do not depend on this smoothness. Using notations in Lemma \ref{L3.1}, denote $\theta_k:=h_k/r_k$. By this lemma, we have $\sum\theta_k<\8$. We can start our considerations with large enough $k\ge 1$. Therefore, without loss of generality, we assume that $0\le \theta_k\le\ve_0<1$ for all $k\ge 1$, where $\ve_0=\ve_0(n,\nu)$ is a small constant in $(0,1)$, which will be specified later.

    For integers $k\ge 1$, denote $Q_k:=Q\cup C_{r_k}$, where
    \[C_r:=\{x=(x',x_n)\in\rn:\; |x'|<r,\;0<x_n<r\}.\]
    We will approximate the given function $v$ by solutions $v_k\in C^2(Q_k)\cap C(\ol{Q_k})$ of the Dirichlet problem
    \[Lv_k=0\itext{in}Q_k,\qquad
    v_k=g_k\itext{on}\pa Q_k,\]
    where $g_k$ is a continuous function on $\pa Q_k$, defined as $g_k\eq v$ on $(\pa Q_k)\cap (\pa Q)$, and $g_k\eq 0$ on the remaining part of $\pa Q_k$. Note that $Q_k$ are Lipschitz domains, hence the existence of such solutions for equations with smooth coefficients is known. It is easy to see that $Q_k\searrow Q$, and by the comparison principle $v_k\searrow v$ in $Q$ as $k\to\8$.

    The following estimate is an important step in our proof:
    \begin{equation}\label{3.3}
        \sup_{C_{r_k}}\frac{v_k}{x_n}\le
        N\cdot\frac{v_k(0,r_k)}{r_k},
        \itext{where} N=N(n,\nu)\ge 1.
    \end{equation}
    Here both functions $v_k$ and $x_n$ are positive and satisfy the equation $Lv=0$ in the domain $\Omega_{2r}:=Q_r\cap C_{2r}$, and $v_k=0$ on the set $\Gamma_{2r}:=(\pa Q_r )\cap (\pa \Omega_{2r})$ with $r=r_k$. However, we cannot apply Corollary \ref{C2.7} directly, because $\Gamma_{2r}$ is not represented as the graph of a Lipschitz function. In order to fix this gap, note that $\Gamma_{2r}$ is a surface of rotation, and the function $\psi(r)$ is non-decreasing. Therefore, $\Gamma_{2r}$ is still the graph of a Lipschitz function \emph{locally} with an absolute constant $K$ in a neighborhood of each of its point $x_0$, in a rotated coordinate system centered at $x_0$. This allows us to estimate the ratio $v_k/x_n$ near $x_0$ by the same ratio at a point strictly inside of $\Omega_{2r}$, an then use the Harnack inequality in order to get \er{3.3} with a constant $N=N(n,\nu)\ge 1$. This argument is similar to that in Remark \ref{R2.6}. In the rest of the proof, $N$ denotes different positive constants depending only on $n$ and $\nu$.

    Next, note that $0\le x_n\le h_k:=\psi(r_k)$ on the set $(\pa Q)\cap \ol{C_k}$, hence by \er{3.3}, $0\le v\le v_k\le N\theta_kv_k(0,r_k)$ on this set. We also have $v=v_k$ on the rest of $\pa Q$. By the comparison principle, this yields
    \begin{equation}\label{3.4}
        0\le v_k-v_{k+1}\le v_k-v\le
        N\theta_kv_k(0,r_k)\itext{in} Q.
    \end{equation}
    Combining the Harnack inequality with Corollary \ref{C2.5}, we get
    \begin{equation}\label{3.5}
    0<\frac{v_k(0,r_k)}{r_k}
    \le \frac{Nv_k(0,r_{r+1})}{r_{k+1}}
    \le N\mu_k,\itext{where}
    \mu_k:=\inf_{C_{r_{k+1}}}\frac{v_k}{x_n}.
    \end{equation}
    Further, from an elementary inequality $\inf A_k-\inf B_k\le \sup (A_k-B_k)$ and $C_{r_{k+2}}\su C_{r_{k+1}}$ it follows
    \[ \mu_k-\mu_{k+1} \le \inf_{C_{r_{k+2}}}\frac{v_k}{x_n}-\inf_{C_{r_{k+2}}}\frac{v_{k+1}}{x_n}
    \le \sup_{C_{r_{k+2}}}\frac{v_k-v_{k+1}}{x_n}.\]
    Here the right side can be estimated by Theorem \ref{T2.4}. In combination with \er{3.4} and \er{3.5}, this gives us
    \[ \mu_k-\mu_{k+1} \le
    \frac{N(v_k-v_{k+1})(0,r_{k+2})}{r_{k+2}}
    \le\frac{N\theta_kv_k(0,r_k)}{r_k}
    \le N\theta_k\mu_k.\]

    As we noticed in the beginning of the proof, we can assume that $\theta_k:=h_k/r_k\le\ve_0$ for all $k$, with a convenient choice of the constant $\ve_0=\ve_0(n,\nu)\in (0,1)$. Choose $\ve_0$ such that in the previous expression, $\alpha_k:=N\theta_k\le N\ve_0\le 1/2$ for all $k$. By iteration, we obtain
    \[ \mu_{k+1}\ge (1-\alpha_k)\mu_k\ge
    (1-\alpha_k)(1-\alpha_{k-1})\cdots(1-\alpha_1)\mu_1.\]
    Finally, we use the fact that convergence of the series $\sum\alpha_j=N\sum\theta_j$ is equivalent to convergence of the product $\prod(1-\alpha_j)$. More specifically, from convexity of the function $f(\alpha):=-\ln(1-\alpha)$ it follows that its values lie between $\alpha$ and $2\ln 2\cdot\alpha$ for all $\alpha\in [0,1/2]$. Hence
    \[ -\ln\mu_{k+1}\le -\ln\mu_1
    -\sum_{j=1}^k\ln(1-\alpha_j)
    \le -\ln\mu_1+2\ln 2\sum_{j=1}^{\8}\alpha_j
    <\8\]
    for all $k$. Then $v_k(0,r_{k+1})/r_{k+1}\ge\mu_k\ge\con>0$ for all $k$, and by the Harnack inequality, same is true for the sequence $v_k(0,r_k)/r_k$. We can also assume that $N\theta_k\le 1/2$ in \er{3.4}, hence $v(0,r_k)/r_k\ge v_k(0,r_k)/2r_k\ge\con>0$ for all $k$.

    Now we see that the ratio $v(0,x_n)/x_n$ is separated from $0$ for $x_n=r_k:=4^{-k},\,k\ge 1$. By the Harnack inequality, this is also true for $r_{k+1}\le x_n\le r_k$, and \er{3.2} follows.
\end{proof}

\emph{Proof of Theorem \ref{T1.8}.}
    As in the preceding proof, we can assume that $a_{ij}$ are smooth. Replacing $r_0>0$ in \er{1.5} by a smaller number if necessary, we can also assume that $u$ is not identically $0$ on $\pa Q$. Choose an arbitrary function $g\in C(\pa Q)$, such that $0\le g\le u$ on $\pa Q$, $g\eq 0$ on $\Gamma:=(\pa Q)\cap \{x_n=\psi(|x'|)\}$, and $g$ is not identically $0$. Then define $v\in C^2(Q)\cap C(\ol{Q})$ as a solution of the equation $Lu=0$ in $Q$ with the boundary data $v=g$ on $\pa Q$. This function $v$ automatically satisfies all the assumptions of Lemma \ref{L3.3}, and moreover, by the comparison principle, $u\ge v>0$ in $Q$. Therefore, for the proof of \er{1.7}, it suffices to establish a similar property for the function $v$.

    Fix an arbitrary vector $\vect{l}=(l',l_n)\in\rn_+$, choose a constant $K_1>0$ such that $l_n>K_1|l'|$, and another constant $K_0\in (0,K_1)$. Finally take a constant $R_0\in (0,r_0]$ according to Corollary \ref{C3.2}. This guarantees that $Q$ contains the set $V_0:=\{|x|<R_0,\,x_n>K_0|x'|\}$. In turn, by our construction $V_0$ contains the set $V_1:=\{|x|<R_0/2,\,x_n>K_1|x'|\}$, and $t\vect{l}\in V_1$ for all $t$ in an interval $(0,t_0)$. By the Harnack inequality, $v(0,tl_n)\le Nv(t\vect{l})$ for all $t\in (0,t_0)$. Now the desired estimate follows from \er{3.2} with $x_n=tl_n$.
\hfill $\Box$
\medskip

In the rest of the paper, we skip some details of proofs which are similar to those in the proofs of Lemma \ref{L3.3} and Theorem \ref{T1.8}. In particular, we assume that $a_{ij}$ are smooth, so that the Dirichlet problem $Lu:=\sum a_{ij}D_{ij}u=0$ in $\Omega$ with the boundary condition $u=g$ on $\po$ has a classical solution for any bounded Lipschitz domain $\Omega$ and any function $g\in C(\po)$. The following lemma covers a ``model'' case for the proof of Theorem \ref{T1.9}.

\begin{lemma}\label{L3.4}
    Let $\psi(r)$ be a non-negative, non-decreasing function on $[0,r_0]$, with $I(\psi)<\8$. Define
    \begin{equation}\label{3.6}
        \begin{split}
      Q^* & := \{|x'|<r_0,\,-\psi(|x'|)<x_n<r_0\},\\
      \Gamma^* & : =(\pa Q^*)\cap \{-x_n\in\psi(|x'|)\}.
    \end{split}
    \end{equation}
    Let $w$ be a function in $C^2(Q^*)\cap C(\ol{Q^*})$, such that
    \[w>0,\quad Lw=0\itext{in}Q^*;\qquad
    w=0 \itext{on} \Gamma^*.\]
     Then the ratio $\,w(x)/|x|\,$ is bounded on $Q^*$. As in Lemma \ref{L3.3}, we assume $\psi(r)=[\psi(r-),\psi(r+)]$ for $0<r<r_0$.
\end{lemma}

\begin{proof}
    We approximate $Q^*$ by a sequence of domains $Q_k^*,\,k\ge 1$, with flat boundaries in the $r_k$-neighborhood of the origin. Namely, set
    \[Q_k^*:=\{x=(x',x_n):\;|x'|<r_0,\;
    -\psi_k(|x'|)<x_n<r_0\},\]
    where $\psi_k(r)\eq 0$ on $[0,r_k]$, and $\psi_k(r)\eq \psi(r)$ on $(r_k,r_0]$. Correspondingly, the given function $w$ will be approximated by solutions $w_k\in C^2(Q_k^*)\cap(\ol{Q_k^*})$ of the Dirichlet problems
    \[ Lw_k=0\itext{in} Q_k^*,\qquad
    w_k=g_k\itext{on}\pa Q_k^*,\]
    where the functions $g_k\in C(\pa Q_k^*)$ are defined by the equalities $\,g_k\eq w\,$ on $(\pa Q_k^*)\cap(\pa Q^*)$, and $g_k\eq 0$ on $(\pa Q_k^*)\sm(\pa Q^*)$. We have $Q_k^*\nearrow Q^*$, and by the comparison principle $w_k\nearrow w$ in $Q^*$ as $k\to\8$, if we formally extend $w_k\eq 0$ on $Q^*\sm Q_k^*$.
    As in the proof of Lemma \ref{L3.3}, we can assume that $\theta_k:=h_k/r_k\le \ve_0=\ve_0(n,\nu)$ - a small constant in $(0,1)$.

    We can apply Corollary \ref{C2.7} to the functions
    \[u:=w,\quad v:=x_n+h_{k-1}\itext{in} D_k:=\{|x'|<r_k,\;-\psi(|x'|)<x_n<r_k\}\]
    in the same way as we did it in the proof of \er{3.3}. These functions are positive, satisfy $Lu=Lv=0$ in a larger domain $D_{k-1}$, and $u:=w=0$ on its ``bottom'' $(\pa D_{k-1})\cap \{-x_n\in\psi(|x'|)\}$. Therefore,
    \[ \sup_{D_k}\frac{w}{x_n+h_{k-1}}
    \le \frac{Nw(0,r_k)}{r_k},\]
    From this estimate it follows
    \[0=w_k\le w\le N\theta_{k-1}w(0,r_k)\itext{on}(\pa Q_k^*)\cap \ol{D_k}.\]
    On the rest of $\pa Q_k^*$, we have $w_k=w$. By the comparison principle,
    \begin{equation}\label{3.7}
    0\le w_{k+1}-w_k \le w-w_k\le N\theta_{k-1}w(0,r_k)
    \itext{in} Q_k^*.
    \end{equation}
    In particular, assuming $N\theta_{k-1}\le N\ve_0\le 1/2$, we get $w(0,r_k)\le 2w_k(0,r_k)$.

    Further, we apply Corollary \ref{C2.7} once again, with $v\eq 1$, and then use the Harnack inequality. This implies
    \begin{equation}\label{3.8}
    \sup_{D_k}w\le Nw(0,r_k)\le Nr_{k+1}M_k,
    \itext{where}
    M_k:=\sup_{C_{r_{k+1}}}\frac{w_k}{x_n}.
    \end{equation}
    Using inequality $\sup A_k-\sup B_k\le \sup (A_k-B_k)$ and Theorem \ref{T2.4} with $u:=w_{k+1}-w_k,\;v:=x_n$ in $C_{r_{k+2}}\su C_{r_{k+1}}$, we obtain
    \[ M_{k+1}-M_k\le \sup_{C_{r_{k+2}}}\frac{w_{k+1}-w_k}{x_n}
    \le \frac{N(w_{k+1}-w_k)(0,r_{k+2})}{r_{k+2}}.\]
    Together with \er{3.7} and \er{3.8}, this implies
    \[ M_{k+1}-M_k \le N\theta_{k-1}w(0,r_k)/r_{k+2}
    \le N\theta_{k-1}M_k,\]
    so that $M_{k+1}\le (1+N\theta_{k-1})M_k$. Iterating this estimate and using the fact that from convergence of the series $\sum\theta_k$ it follows convergence of the product $\prod(1+N\theta_{k-1})$, we get the estimate $M_k\le NM_1$ for all $k\ge 1$. Finally, in order to prove the boundedness of $w(x)/|x|$, it suffices to show that its supremum over the set $Q^*\cap\{r_{k+1}<|x|\le r_k\}$, which is a subset of $D_k$, does not exceed a constant uniformly for all $k$. This is an immediate consequence of \er{3.8}: for each $x$ in this set,
    \[\frac{w(x)}{|x|}
    \le \frac{1}{r_{k+1}}\cdot\sup_{D_k}w
    \le NM_k\le NM_1<\8.\]
    Lemma is proved.
\end{proof}

\emph{Proof of Theorem \ref{T1.9}.}
    From our assumptions it follows that the set  is a subset of $Q^*$ defined in \er{3.6}. Replacing $r_0>0$ by a smaller number if necessary, we can assume that $u=0$ on $(\po)\cap Q^*$. Then the function $g$ on $\pa Q^*$ defined by the equalities $g\eq u$ on $(\pa Q^*)\cap\Omega$, and $g\eq 0$ on $(\pa Q^*)\sm \Omega$, belongs to $C(\pa Q^*)$. Assuming that $a_{ij}$ are smooth, we can define $w\in C^2(Q^8)\cap C(\ol{Q^*})$ as a solution to the equation $Lu=0$ in $Q^*$ with the boundary condition $w=g$ on $\pa Q^*$. By the comparison principle, $0<u\le w$ in $Q^*\cap\Omega$. Therefore, $u(x)/|x|$ is bounded in $\Omega\cap B_{r_0}(0)$ by Lemma \ref{L3.4}
\hfill $\Box$
{}

\end{document}